\documentclass[12pt]{amsart}

\usepackage{amsmath}
\usepackage{amsfonts}
\usepackage{amssymb}
\usepackage[utf8]{inputenc}
\usepackage[T1]{fontenc}

\newtheorem{thrm}{Theorem}[section]

   \newtheorem{fact}[thrm]{Proposition}
   \newtheorem{lemma}[thrm]{Lemma}
   \newtheorem{col}[thrm]{Corollary}
   \newtheorem{defn}[thrm]{Definition}
   \newtheorem{remark}[thrm]{Remark}
   
   \newtheorem{problem}[thrm]{Problem}

\newenvironment{proofclaim}
{\textit{\\Proof.}}{\begin{flushright}
$\lozenge$
\end{flushright}}

\newenvironment{claim}{\textbf{\\\\Claim.}}

\DeclareMathOperator{\Fin}{Fin}

\title{On the t-equivalence relation}

\author[M.\ Krupski]{Miko\l aj Krupski}
\address{Institute of Mathematics\\ Polish Academy of Sciences\\ \newline Ul. \'Sniadeckich 8\\00--956 Warszawa\\ Poland }
\email{krupski@impan.pl}



\begin{document}
\baselineskip=17pt

\begin{abstract}
For a completely regular space $X$, denote by $C_p(X)$ the space of continuous real-valued functions on $X$, with the pointwise convergence topology.
In this article we strengthen a theorem of O.\ Okunev concerning preservation of some topological properties of $X$ under homeomorphisms
of function spaces $C_p(X)$. From this result we conclude new theorems similar to results
of R.\ Cauty and W.\ Marciszewski about preservation of certain dimension-type properties of spaces $X$
under continuous open surjections between function spaces $C_p(X)$.
\end{abstract}

\maketitle

\section{Introduction}
One of the main objectives in the theory of $C_p(X)$ spaces is to classify spaces of continuous functions up to
homeomorphisms. One can do this by investigating which topological properties of a space $X$ are shared with a space
$Y$, provided $X$ and $Y$ are {\em $t$-equivalent}, i.e. $C_p(X)$ and $C_p(Y)$ are homeomorphic. 
Recently, O.~Okunev published a paper \cite{O} in which he found some new topological invariants of the $t$-equivalence relation.
All of them are obtained from the following, very interesting Theorem (see \cite[Theorem 1.1]{O})
\begin{thrm}(Okunev)\label{Thrm-O}
Suppose that there is an open continuous surjection from $C_p(X)$ onto $C_p(Y)$. Then there are spaces $Z_n$ ,
locally closed subspaces $B_n$ of $Z_n$, and locally closed
subspaces $Y_n$ of $Y$, $n\in \mathbb{N}^+$, such that each $Z_n$ admits a perfect finite-to-one mapping onto a closed subspace of
$X^n$, $Y_n$ is an image under a perfect mapping of $B_n$, and $Y=\bigcup\{Y_n:\; n\in\mathbb{N}^+\}$.
\end{thrm}
In the formulation of the above theorem in \cite{O} the assumption about the existence of an open continuous surjection is replaced by the
assumption that these function spaces are homeomorphic. However, as noticed in \cite[remarks at the end of section 1]{O} a careful analysis
of the proof reveals that the weaker assumption is sufficient.  
In this paper we will discuss the proof of the above theorem (detailed proof can be found in \cite{O}). Then using an idea from
\cite{M} we will show how to slightly improve Okunev's result, answering Question 1.9 from \cite{O}.
In the subsequent sections we will derive a few corollaries from strengthened form of Okunev's theorem.
We will use it to find new invariants of the $t$-equivalence relation concerning dimension. These results are in the spirit
of the significant theorems of R.\ Cauty from \cite{C} and W.\ Marciszewski from \cite{M}.

We should also mention here, that the answer to Question 2.12 posed in \cite{O} is known (see \cite{BH}, \cite{G}).
Thus one can show (see \cite{O})
that $\sigma$-discreteness is preserved by the $t$-equivalence relation (see \cite[Question 2.9]{O}).
In fact, from a result of Gruenhage from \cite{G} one can conclude more,
namely that $\kappa$-discreteness is preserved by the relation of $t$-equivalence (see Theorem \ref{k-discreteness} below).
We discuss this in Section 3.\\

Unless otherwise stated, all spaces in this note are assumed to be Tychonoff.
For a space $X$ we denote by $C_p(X)$ the space of continuous, real-valued
functions on $X$ with the pointwise convergence topology. We say that spaces $X$ and $Y$ are {\em $t$-equivalent}, provided $C_p(X)$ and
$C_p(Y)$ are homeomorphic. The subspace of a topological space is {\em locally closed} if it is the intersection of a closed set and an open set. The mapping
$\varphi:X\rightarrow Y$ between topological spaces is {\em perfect}, provided it is closed and all fibers $\varphi^{-1}(y)$ are compact.
For a space $X$ we denote by $\Fin(X)$ the hyperspace of all finite subsets of $X$ with the Vietoris topology.
We follow Engelking's book \cite{E} regarding dimension theory.

\section{On a result of Okunev}

The main goal of this section is to answer Question 1.9 from \cite{O}, i.e. to prove that in the statement of Theorem \ref{Thrm-O}
we may additionally require that for every $n\in \mathbb{N}^+$ the space
$Y_n$ is in fact an image under a perfect finite-to-one mapping of $B_n$.
To this end we need to discuss the main ideas from \cite{O}. For the convenience of the reader our notation will be almost the same as in \cite{O}.

The real line $\mathbb{R}$ is considered as a subspace of its two-point compactification $I=\mathbb{R}\cup \{-\infty,+\infty\}$.
For a continuous function $f:Z\rightarrow\mathbb{R}$, the function $\widetilde{f}:\beta Z\rightarrow I$ is the continuous extension of $f$.
For every $n\in \mathbb{N}^+$, $\overline{z}=(z_1, \ldots, z_n)\in (\beta Z)^n$ and $\varepsilon>0$ we put
$$O_Z(\overline{z};\varepsilon)=O_Z(z_1,\ldots,z_n;\varepsilon)=\{f\in C_p(Z):\lvert \widetilde{f}(z_1)\rvert < \varepsilon,\ldots ,
\lvert \widetilde{f}(z_n)\rvert < \varepsilon\}.$$
Similarly, for every $A\in \Fin(Z)$ and $\varepsilon>0$ we put
$$O_Z(A;\varepsilon)=\{f\in C_p(Z):\forall z\in A\;\;\lvert f(z) \rvert<\varepsilon\}.$$
For a point $z\in Z$ we put
$$\overline{O}_Z(z;\varepsilon)=\{f\in C_p(Z):\lvert f(z)\rvert\leqslant\varepsilon\}.$$
Let $\Phi:C_p(X)\rightarrow C_p(Y)$ be an open surjection which takes the zero function on $X$ to the zero function on $Y$ (we can assume this since
$C_p(X)$ and $C_p(Y)$ are homogeneous).
For every $(m,n)\in \mathbb{N}^+\times\mathbb{N}^+$ we put
$$Z_{m,n}=\{(\overline{x},y)\in X^n\times Y:\Phi(O_X(\overline{x};\tfrac{1}{m}))\subseteq \overline{O}_Y(y;1)\}.$$
By $\pi_X:X^n\times \beta Y\rightarrow X^n$ we denote the projection and we put
$$p_{m,n}=\pi_X\upharpoonright Z_{m,n}:Z_{m,n}\rightarrow X^n.$$
Similarly, by $\pi_{\beta Y}:(\beta X)^n\times \beta Y\rightarrow \beta Y$ we denote the projection and we put
$$A_{m,n}=\pi_{\beta Y}(Z_{m,n}).$$
Denote by $S_{m,n}$ the closure of $Z_{m,n}$ in $(\beta X)^n\times \beta Y$.
For every $m\in \mathbb{N}^+$ we put $Y_{m,1}=A_{m,1}$ and for every $n>1$, $Y_{m,n}=A_{m,n}\setminus A_{m,n-1}$.
Finally let us put $B_{m,n}=S_{m,n}\cap\pi_{\beta Y}^{-1}(Y_{m,n})$ and let
$$r_{m,n}=\pi_{\beta Y}\upharpoonright B_{m,n}:B_{m,n}\rightarrow Y_{m,n}.$$

The following properties
are satisfied (see \cite{O}):
\begin{enumerate}
 \item[(0)] the set $Z_{m,n}$ is closed in $X^n\times\beta Y$;
 \item[(1)] $p_{m,n}$ maps perfectly $Z_{m,n}$ onto a closed subset of $X^n$;
 \item[(2)] the mapping $p_{m,n}$ is finite-to-one;
 \item[(3)] the sets $A_{m,n}$ are closed, thus the sets $Y_{m,n}$ are locally closed;
 \item[(4)] $Y=\bigcup_{m,n\in\mathbb{N}^+}Y_{m,n}$;
 \item[(5)] the set $B_{m,n}$ is locally closed in $Z_{m,n}$;
 \item[(6)] the mapping $r_{m,n}$ is perfect;
\end{enumerate}
Clearly, Theorem \ref{Thrm-O} follows from (1)--(6).

We will use the following version of the $\Delta$-system Lemma which can be easily proved by induction (see also \cite[A.1.4]{vM})
\begin{fact}\label{delta-system}
Let $X$ be a set, let $n\in\mathbb{N}^+$ and let $\mathcal{A}$ be an infinite collection of subsets of $X$ each of cardinality $\leqslant n$.
Then there is $A_0\subseteq X$ with $\lvert A_0\rvert<n$ and a sequence $A_1, A_2, \ldots$ of distinct elements of $\mathcal{A}$ such that
for distinct $i,j\geqslant 1$ we have $A_i\cap A_j=A_0$.
\end{fact}

Now we are ready to prove the following strengthening of Theorem \ref{Thrm-O}.
\begin{thrm}\label{thrm-gen}
Suppose that there is an open continuous surjection $\Phi$ from $C_p(X)$ onto $C_p(Y)$. Then there are spaces $Z_n\subseteq X^n\times Y$,
locally closed subspaces $B_n$ of $Z_n$, and locally closed
subspaces $Y_n$ of $Y$, $n\in \mathbb{N}^+$, such that each $Z_n$ admits a perfect finite-to-one mapping onto a closed subspace of
$X^n$, $Y_n$ is an image under a perfect {\bf\em finite-to-one} mapping of $B_n$, and $Y=\bigcup\{Y_n:\; n\in\mathbb{N}^+\}$.
\end{thrm}
\begin{proof}
It is enough to prove that
\begin{enumerate}
 \item[(7)] the mapping $r_{m,n}$ is finite-to-one.
\end{enumerate}
To this end let us put
$$Z'_{m,n}=\{(A,y)\in \Fin(X)\times Y:\; \lvert A \rvert\leqslant n \text{ and } \Phi(O_X(A;\tfrac{1}{m}))\subseteq
\overline{O}_Y(y;1)\}.$$
The natural mapping $h:Z_{m,n}\rightarrow Z'_{m,n}$ defined by
$$h((x_1,\ldots, x_n),y)=(\{x_1,\ldots,x_n\},y),$$ is finite-to-one. Hence, if the set $\{A\in\Fin(X):\;(A,y)\in Z'_{m,n}\}$ is finite,
then the set
$\{\overline{x}\in X^n:\;(\overline{x},y)\in Z_{m,n}\}$ is also finite.
We will prove that this is the case.
\begin{claim}
For any $y\in Y_{m,n}$ the set $\{A\in\Fin(X):\;(A,y)\in Z'_{m,n}\}$ is finite.
\end{claim}
\begin{proofclaim}
This is basically \cite[Lemma 3.4]{M}. 
Assume the contrary. Then by Proposition \ref{delta-system}, there exists $A_0\in \Fin(X)$ and a sequence $A_1,A_2,\ldots$ of finite subsets of $X$
such that $\lvert A_0\rvert< n$, for distinct $i,j\geqslant 1$ we have $A_i\cap A_j=A_0$ and $(A_i,y)\in Z'_{m,n}$ for each $i\geqslant1$. 

To end the proof of the Claim we need to show $(A_0,y)\in Z'_{m,n}$. Indeed, then we would have $(A_0,y)\in Z'_{m,n-1}$
(since $\lvert A_0\rvert<n$) so $y\in A_{m,n-1}$ contradicting the assumption $y\in Y_{m,n}=A_{m,n}\setminus A_{m,n-1}$.

Let $f\in O_X(A_0;\tfrac{1}{m})$. We need to show that $\lvert\Phi(f)(y)\rvert\leqslant 1$. Assume the contrary. The set
$\Phi^{-1}(\{\varphi\in C_p(Y):\lvert \varphi(y)\rvert>1\})$ is an open neighborhood of $f$. Hence, there exists a finite set $B\in \Fin(X)$
and a natural number $k\in \mathbb{N^+}$ such that for any $g\in C_p(X)$ if $(f-g)\in O_X(B;\tfrac{1}{k})$, then $\lvert \varphi(g)(y)\rvert>1$.

For $i\geqslant 1$, the sets $A_i\setminus A_0$ are pairwise disjoint. Hence, there exists $i\geqslant 1$ such that
$B\cap (A_i\setminus A_0)=\emptyset$. Take $g\in C_p(X)$ satisfying
$$g\upharpoonright (A_0\cup B)=f\upharpoonright (A_0\cup B) \text{ and } g\upharpoonright (A_i\setminus A_0)\equiv 0.$$
Then $g\in O_X(A_i;\tfrac{1}{m})$ so $\lvert \varphi(g)(y)\rvert\leqslant 1$. On the other hand $(f-g)\in O_X(B;\tfrac{1}{k})$ so
$\varphi(g)(y)>1$, a contradiction.
\end{proofclaim}
For any $y\in Y_{m,n}$, we have $r_{m,n}^{-1}(y)\subseteq \{\overline{x}\in X^n:\;(\overline{x},y)\in Z_{m,n}\}.$
The latter set is, as we proved, finite so the mapping $r_{m,n}$ is finite-to-one.
\end{proof}

Theorem \ref{thrm-gen} answers Question 1.9 from \cite{O}.

\section{$\kappa$-discreteness}
Recall, that a space is called {\em $\kappa$-discrete} ({\em $\sigma$-discrete}) is it can be represented as a union of at most $\kappa$ many
(countably many) discrete subspaces. In \cite{O}, O.\ Okunev asked if $\sigma$-discreteness is preserved by the $t$-equivalence relation
(see \cite[Question 2.9]{O}). He also showed how to reduce this question to the following one: {\em Is a perfect image of a $\sigma$-discrete space
also $\sigma$-discrete?} However, the affirmative answer to this question is known (see \cite{BH}, \cite{G}).
G.\ Gruenhage proved even a stronger result that, for any infinite cardinal $\kappa$, a perfect image of a $\kappa$-discrete space
is $\kappa$-discrete. Since the reduction made by Okunev works also for $\kappa$-discrete spaces, we have the following theorem.
\begin{thrm}\label{k-discreteness}
If there is an open continuous surjection from $C_p(X)$ onto $C_p(Y)$ and $X$ is $\kappa$-discrete, then
$Y$ is $\kappa$-discrete.
\end{thrm}

\section{The property $C$}

From Theorem \ref{thrm-gen} we can conclude some new results concerning the behavior of dimension under the $t$-equivalence relation.
The main motivation for this is the following, famous in $C_p$-theory problem concerning dimension (see e.g. \cite[Problem 20 (1045)]{A} or
\cite[Problem 2.9]{M2}).
\begin{problem}(Arkhangel’skii)
Suppose $X$ and $Y$ are t-equivalent. Is it true that $\dim X=\dim Y$? 
\end{problem}
It is well known, that if we additionally assume that $C_p(X)$ and $C_p(Y)$ are {\em linearly} or {\em uniformly} homeomorphic the above
problem has an affirmative answer (see \cite{M2}).
In general, very little is known about the behavior of dimensions under the relation of $t$-equivalence.
We do not know for example if the spaces $C_p(2^\omega)$ and $C_p([0,1])$ or the spaces $C_p([0,1])$ and $C_p([0,1]^2)$ are homeomorphic
(see \cite{M2}).

We should recall the following two definitions (see \cite{E} and \cite{F}).

\begin{defn}
A normal space $X$ is called a $C$-space if, for any sequence of its open covers $(\mathcal{U}_i)_{i\in\omega}$,
there exists a sequence of disjoint families $\mathcal{V}_i$ of open sets such that $\mathcal{V}_i$ is a refinement of $\mathcal{U}_i$
and $\bigcup_{i\in\omega}\mathcal{V}_i$ is a cover of $X$.
\end{defn}

\begin{defn}
A normal space $X$ is called a $k$-$C$-space, where $k$ is a natural number $\geqslant 2$, if
for any sequence of its covers $(\mathcal{U}_i)_{i\in \omega}$ such that each cover $\mathcal{U}_i$ consists of at most $k$ open sets,
there exists a sequence of disjoint families $(\mathcal{V}_i)_{i\in \omega}$ of open sets such that for every $i\in \omega$ the family
$\mathcal{V}_i$ is a refinement of $\mathcal{U}_i$ and $\bigcup_{i\in\omega}\mathcal{V}_i$ is a cover of $X$.
\end{defn}

It is known that a normal space is {\em weakly infinite-dimensional} if and only if it is a $2$-$C$-space (see \cite{F}).
It is clear that we have the following sequence of inclusions
$$\text{weakly infinite-dimensional}=2\text{-}C \supseteq 3\text{-}C \supseteq \ldots$$
and that any $C$-space is a $k$-$C$-space for any $k\in\{2,3,\ldots\}$.

R.\ Cauty proved in \cite{C} the following theorem concerning weak infinite dimension.
\begin{thrm}(Cauty)\label{Cauty}
Let $X$ and $Y$ be metrizable compact spaces such that $C_p(Y)$ is an image of $C_p(X)$ under a continuous open mapping.
If for all $n\in\mathbb{N}^+$ the space $X^n$ is weakly infinite-dimensional, then for all $n\in\mathbb{N}^+$ the finite power $Y^n$
is also weakly infinite-dimensional.
\end{thrm}

Using Theorem \ref{thrm-gen} we can prove a version of the above theorem of Cauty for $k$-$C$-spaces.
We need a suitable lemma, which is a version of \cite[Theorem 4.1]{P}.
\begin{lemma}\label{lemma}
Suppose that $K$ and $L$ are compact metrizable spaces. Let $f:K\rightarrow L$ be a continuous countable-to-one surjection.
If $L$ is a $k$-$C$ space, then so is $K$. 
\end{lemma}
\begin{proof}
From the proof of Theorem 4.1 in \cite{P}, it follows that it suffices to check that a class of $\sigma$-compact metrizable
$k$-$C$-spaces is admissible, i.e. satisfies the following four conditions
\begin{itemize}
 \item[(i)] if $X$ is a $k$-$C$-space and $Y$ is homeomorphic to a closed subspace of $X$, then $Y$ is a $k$-$C$-space;
 \item[(ii)] a space which is a countable union of $k$-$C$-spaces is a $k$-$C$-space;
 \item[(iii)] if $f:X\rightarrow Y$ is a perfect mapping, $Y$ is zero-dimensional and all fibers $f^{-1}(y)$ are $k$-$C$-spaces,
then $X$ is a $k$-$C$-space;
 \item[(iv)] if $A\subseteq X$, $A$ is a $k$-$C$-space and all closed subsets of $X$ disjoint from $A$ are $k$-$C$-spaces, then
$X$ is a $k$-$C$-space.
\end{itemize}
Condition (i) is \cite[Proposition 2.13]{F}. Condition (ii) is \cite[Theorem 2.16]{F}. Condition (iii) is \cite[Theorem 5.2]{F}.
Condition (iv) is actually \cite[Lemma 2]{GR} (although it deals with $C$-spaces, its proof works also for $k$-$C$-spaces).
\end{proof}

\begin{thrm}\label{thrm-Cauty-gen}
Let $X$ and $Y$ be metrizable $\sigma$-compact spaces such that $C_p(Y)$ is an image of $C_p(X)$ under a continuous open mapping.
Fix a natural number $k\geqslant2$. If for all $n\in\mathbb{N}^+$ the space $X^n$ is a $k$-$C$-space, then
$Y$ is also a $k$-$C$-space.
\end{thrm}

\begin{proof}
We apply Theorem \ref{thrm-gen} as follows. Let $Y_n$, $Z_n$, $B_n$ be as in the statement of Theorem \ref{thrm-gen}.
The space $Z_n\subseteq X^n\times Y$ is metrizable and $\sigma$-compact. Indeed, it is easy to check that a perfect preimage of a compact set
is compact, so from $\sigma$-compactness of $X$ follows $\sigma$-compactness of $Z_n$. Let $Z_n=\bigcup_{m=1}^\infty K_m$, where each $K_m$ is
compact.

Since $Z_n$ is a perfect finite-to-one preimage of a closed subspace of $X^n$ and a closed subspace of a metrizable $k$-$C$-space is a $k$-$C$-space
(see \cite[1.15 and 2.19]{F}), each $K_m$ is a $k$-$C$-space by Lemma \ref{lemma}. Since a countable union of closed
$k$-$C$-subspaces is a $k$-$C$-space (see \cite[2.16]{F}), we get that
$Z_n$ is a $k$-$C$-space and thus $B_n$ is such (as an $F_\sigma$ subspace of a metrizable $k$-$C$-space \cite[1.15 and 2.19]{F}).

Since the image of a metrizable $k$-$C$-space under a closed mapping with fibers of cardinality $<\mathfrak{c}$ is a
$k$-$C$-space (see \cite[6.17]{F}),
the space $Y_n$ is a $k$-$C$-space for any $n\in\mathbb{N}^+$.
Finally, since the property of being a $k$-$C$-space is invariant with respect to countable unions with closed summands (see \cite[2.16]{F}), we
get that $Y$ is a $k$-$C$-space.
\end{proof}

From the above theorem we can conclude a result very similar to Theorem \ref{Cauty} of R.\ Cauty we mentioned.

\begin{col}
Let $X$ and $Y$ be $\sigma$-compact metrizable spaces such that $C_p(Y)$ is an image of $C_p(X)$ under a continuous open mapping. If
for all $n\in\mathbb{N}^+$ the space $X^n$ is weakly infinite-dimensional, then $Y$ is also weakly infinite-dimensional.
\end{col}

\begin{proof}
Apply Theorem \ref{thrm-Cauty-gen}
with $k=2$.
\end{proof}

Using the same technique, we can prove a similar theorem about $C$-spaces.

\begin{thrm}\label{thrm-C-spaces}
Let $X$ and $Y$ be $\sigma$-compact metrizable spaces. Suppose, that $C_p(Y)$ is an image of $C_p(X)$ under a continuous open mapping.
If $X$ is a $C$-space, then $Y$ is also a $C$-space.
\end{thrm}
\begin{proof}
Since the finite product of compact metrizable $C$-spaces is a $C$-space (see \cite[Theorem 3]{R}) and since being a $C$-space is invariant
with respect to countable unions with closed summands (see \cite[2.24]{F}), the space $X^n$ is a $C$-space for every
$n\in \mathbb{N}^+$.

We apply Theorem \ref{thrm-gen} as in the proof of Theorem \ref{thrm-Cauty-gen}.
Let $Y_n$, $Z_n$, $B_n$ be as in the statement of Theorem \ref{thrm-gen}.

It is known that within the class of metrizable spaces, the property of being a $C$-space is invariant with respect to
$F_\sigma$ subspaces (see \cite[2.25]{F}) and preimages under continuous mappings with fibers being $C$-spaces (see \cite[5.4]{F}). Hence
the space $Z_n$ is a $C$-space and so is $B_n$. It is also known that for compact spaces property $C$ is preserved by
continuous mappings with fibers of cardinality $<\mathfrak{c}$ (see \cite[6.4]{F}). Thus from the $\sigma$-compactness of $Z_n$ (see the proof
of Theorem \ref{thrm-Cauty-gen}) and the fact that a countable union of closed $C$-spaces is a $C$-space (see \cite[2.24]{F}), we conclude that $Y_n$
is a $C$-space. By \cite[2.24]{F} $Y=\bigcup_n Y_n$ is a $k$-$C$-space.
\end{proof}

\section{Countable-dimension}

Let us recall the following definition
\begin{defn}
A space $X$ is countable-dimensional if $X$ can be represented as a countable union of finite-dimensional subspaces.
\end{defn}

It is well known that every countable-dimensional metrizable space is a $C$-space.
In \cite{M} W.\ Marciszewski modifying a technique from \cite{C} proved the following
\begin{thrm}(Marciszewski)\label{Ma}
Suppose that $X$ and $Y$ are t-equivalent metrizable spaces. Then $X$ is countable dimensional if and only if $Y$ is so.
\end{thrm}
As in the previous section, we can use Theorem \ref{thrm-gen} to prove a slightly more general result.
\begin{thrm}\label{thrm-strong}
Let $X$ and $Y$ be metrizable spaces.
Suppose, that $C_p(Y)$ is an image of $C_p(X)$ under a continuous open mapping.
If $X$ is countable-dimensional, then so is $Y$.
\end{thrm}
\begin{proof}
Since $X$ is countable-dimensional and metrizable, every finite power $X^n$ is countable-dimensional (see \cite[Theorem 5.2.20]{E}).
It is also known that within the class metrizable space, countable-dimensionality is invariant with respect to:
preimages under closed mappings with finite-dimensional fibers \cite[Proposition 5.4.5]{E}, subspaces \cite[5.2.3]{E}, images
under closed finite-to-one mappings \cite[Theorem 5.4.3]{E}) and countable unions \cite[5.2.8]{E}.
Thus it is enough to apply Theorem \ref{thrm-gen}.
\end{proof}

\begin{remark}
Theorems \ref{thrm-Cauty-gen}, \ref{thrm-C-spaces}, \ref{thrm-strong} cannot be concluded directly from Theorem \ref{Thrm-O}.
Let us observe that if we take $X=[0,1]$, $Z_n=B_n=[0,1]^n$ and $Y=Y_n=[0,1]^\omega$, then the thesis of Theorem \ref{Thrm-O} holds.
Indeed, in that case $Z_n$ maps onto $X^n$ by a perfect finite-to-one mapping (the identity) and $B_n$ maps onto $Y_n$ perfectly, so
Okunev's theorem from \cite{O} (Theorem \ref{Thrm-O}) does not prove that spaces $[0,1]$ and $[0,1]^\omega$ are not $t$-equivalent.
To conclude the latter, we need to use the fact that the existence of a continuous open surjection between
$C_p(X)$ and $C_p(Y)$ implies that $Y_n$ is an image of $B_n$ under a finite-to-one mapping.
\end{remark}

\textbf{Acknowledgment.}
\medskip

The author is indebted to Witold Marciszewski for valuable comments and remarks.

\end{document}